\documentclass[12pt]{amsart}
\usepackage{amssymb}
\usepackage[latin1]{inputenc}
\usepackage{amsthm}
\usepackage{amsfonts}
\usepackage{mathrsfs}
\usepackage{tikz-cd}

\usepackage{graphicx}
\usepackage{amsmath}
\usepackage[all]{xy}

\let\mathcal\mathscr

\def\C{{\mathbb C}}

\def\Q{{\mathbb Q}}

\newtheorem{thm}{Theorem}[section]

\def\N{{\bf N}} 
\def\P{{\bf P}}

\def\R{{\bf R}}
\def\Q{{\bf Q}}
\def\C{{\bf C}}

\def\Spec{\mathop{\rm Spec}\nolimits}

\def\tilde{\widetilde}

\def\phi{\varphi}

\def\cC{{\mathcal C}}

\def\cL{{\mathcal L}}
\def\cM{{\mathcal M}}

\def\cO{{\mathcal O}}

\def\cX{{\mathcal X}}

\numberwithin{equation}{section}

\newtheorem{theorem}[thm]{Theorem}

\newtheorem{definition}[thm]{Definition}

\newtheorem{lemma}[thm]{Lemma}

\newtheorem{proposition}[thm]{Proposition}

\theoremstyle{remark}
\newtheorem{rem}[thm]{Remark}

\begin{document}

\title{Rational points of bounded height on entire curves}

\author{Carlo Gasbarri}\address{Carlo Gasbarri, IRMA, UMR 7501
 7 rue Ren\'e-Descartes
 67084 Strasbourg Cedex }
 \email{gasbarri@math.unistra.fr}

\date\today

\maketitle

\begin{abstract}  Let $X$ be an affine or a projective variety defined over a number field $K$ and $\varphi:\C\to X(\C)$ be a holomorphic map with Zariski dense image. We estimate the number of rational points of height bounded by $H$ in the image of a disk of radius $r$ in terms of the the Nevanlinna characteristic function of $\varphi$ and $H$ in a way which generalize the classical Bombieri--Pila estimate to expanding domains. In general this bound is exponential but we show that for many values of $H$ and $r$, the bound is polynomial. 
\end{abstract}

\tableofcontents

\section {Introduction}

\

In the classical text \cite{La1}  S. Lang attributes to Schneider a variant of the following theorem (Thm. 2, page 12 of loc. cit.) and considers it one of the first attempts to axiomatize the theory of transcendence:

\begin{theorem} \label{lang} (Lang) $Let K$ be a number field. Let $F:\C\to\C^2$ be an holomorphic map of finite order $\rho$. Suppose we can find a sequence of subsets $S_n\subset S_{n+1}\subset\C$ and  constants $C>0$ and $\lambda>2$ such that
\begin{itemize}
\item[--] For every $z\in S_n$ we have $\vert z\vert \leq C\cdot n$;
\item[--] For every $n$ and every $z\in S_n$, we have $f(z)\in K^2$ and $h(f(z))\leq Cn^\rho$ (where $h(\cdot)$ is the logarithmic Weil height).
\item[--] $Card(S_n)\sim n^{\lambda\rho}$.
\end{itemize}
Then the image of $F$ is not $K$--Zariski dense (the image is contained in a curve defined over $K$).
\end{theorem}

More recently, E. Bombieri and J. Pila in the paper \cite{BOMBIERIPILA} proved a Theorem which estimate the number of points of bounded height on a relatively compact, Zariski dense Riemann surface in a projective variety defined over a number field. We propose here a version of their Theorem, stated and proved in \cite{gas1}: 

\begin{theorem}\label{BPil} (Bombieri--Pila) Let $M$ be a Riemann surface and $U\subset M$ a relatively compact open set. Let $X$ be a projective variety of dimension $d>1$ defined over a number field $K$ and $f:M\to X(\C)$ be an holomorphic map with Zariski dense image. Let $\epsilon >0$ and, for every positive number $H$ denote by $C_f(H)$ the cardinality of the $z\in U$ such that $f(z)\in X(K)$ and $h(f(z))\leq H$, where $h(\cdot)$ is a logarithmic Weil height on $X$ associated to an ample line bundle. Then
\begin{equation}
C_f(H)\ll_\epsilon \exp(\epsilon\cdot H).
\end{equation}
\end{theorem}

These two theorems have some similarities:
\begin{itemize}
\item[--] Both theorems deal with rational points of analytic Riemann surfaces contained in  varieties defined over number fields.
\item[--] In the hypothesis of both theorem there is a condition on the height of the rational points contained in the Riemann surfaces.
\end{itemize}
But we also notice some big differences:
\begin{itemize}
\item[--] Theorem \ref{lang} deals with rational points on a Riemann surface which is not relatively compact while Theorem \ref{BPil} deals with a relatively compact one.
\item[--] Theorem  \ref{lang} is stated as a algebraization criterion in terms of the number of rational points on the image of an holomorphic function, while Theorem \ref{BPil} is stated as an estimate of the number of rational points contained in a Riemann surface.
\end{itemize}

The objective of this paper is to provide a unified approach to these theorems, demonstrating how they can be considered as parts of the same  theory.

In a optimal generalization of both theorems, we would deal with holomorphic maps from a non compact Riemann Surface to a projective variety and estimate of the number of rational points of bounded height in the image. But at the moment, this seems very difficult. 

Even when we deal with holomorphic maps of $\C$ in a projective variety, some difficulties arise: an important tool of the theory is the Schwartz Lemma, or its refined form which is the First Main Theorem of Nevanlinna Theory. We need to apply it to estimate the norm of sections of a line bundle on points of the image of a suitable disk. This estimate is not efficient if the point we are dealing with is near to the border of the disk itself. As explained in Section \ref{basepoints}, the main tool of the First Main Theorem is the characteristic  function of an holomorphic map. This characteristic function depends on the choice of a base point (which is the point where we need to bound the norm of the section). In order to obtain estimates in terms of a fixed characteristic function, we need to compare these when we change the base points. This is easy when the points are very near but less easy when one of the point is near to border of the disk. This problem is studied in Section \ref{basepoints} where a satisfactory solution is proposed in the case when $X$ is affine, cf Proposition \ref{comparing T}. When $X$ is projective, we find a (quite more involved to obtain) solution for points which are, in the spirit of Cartan estimates, outside of a union of disks the sum of the length of radii is very small, cf. Theorem \ref{comparing T1}.

We state now the main results of this paper by starting with the affine versions of these: Let $X$ be a smooth affine  variety defined over a number field $K$ of dimension $n>1$, we suppose that a smooth projective compactification of it is equipped with an arithmetic polarization  $M$ (as explained in Section \ref{basic arakelov}). We fix an embedding $\sigma :K\hookrightarrow \C$ and an analytic map $\varphi : \C\to X_{\sigma}(\C)$. For every $r>0$, denote by $\Delta_r\subset\C$ the disk of radius $r$. 

For every $(r,H)\in\R^2$, denote by $C_\varphi(r,H)$ the cardinality of the set $S_\varphi(r,H):=\{ w\in\Delta_r\;/\; \varphi(w)\in X(K)\;{\rm and}\; h_M(\varphi(w))\leq H\}$.

The first result of this paper is a generalization of the Bombieri--Pila Theorem to this situation: 

\begin{theorem}\label{BPaffine intro}(cf. Theorem \ref{BPaffine}) In the hypothesis above, let $\epsilon >0$, then
\begin{equation}
C_\varphi(r,H)\ll \exp(\epsilon(H+T_{\varphi, M}(1+\epsilon)r))
\end{equation} where $T_{\varphi, M}(r))$ is the Nevanlinna counting function of the analytic map $\varphi$ and  the constants involved in $\ll$ depend only on $X$, $M$, $\epsilon$ and $\varphi$ but they are independent on $H$ and $r$.
\end{theorem}

One should notice that, if we fix $r$, we find the Bombieri--Pila Theorem \ref{BPil}.

There are examples which show that Theorem \ref{BPaffine intro}  is optimal, in the sense that we cannot hope that, for general $\varphi$, the growing of $C_\varphi(r,H)$ is less than exponential. Nevertheless, in section \ref{polynomial bounds} we show that for many values of $(r,H)$ the number $C_\varphi(r,H)$ is polynomially bounded. More precisely, if $\epsilon>0$ and $\gamma>0$,we will denote by $L(\varphi, \epsilon, \gamma)\subset\R^2$ the following set
 \begin{equation}
 L(\varphi, \epsilon):=\left\{ (T_{\varphi, M}(r),H)\in\R^2\;\;/\;\; C_\varphi(r,H)\leq \epsilon\cdot(T_{\varphi,M}((1+\epsilon)r)+H)^\gamma\right\}.
 \end{equation}
 We will prove then 
 
 \begin{theorem}\label{poly bound thm intro}(Cf. Theorem \ref{poly bound thm})  If $\gamma>{{n}\over{n-1}}$, then, for every $\epsilon >0$,  the set $L(\varphi, \epsilon)$ contains infinitely many disks of arbitrarily big radius.
 \end{theorem}
 
 Actually Theorem \ref{poly bound thm} is more precise than the statement above. We refer to section  \ref{polynomial bounds} for details. 

As explained before, when we consider an entire Zariski dense curve in a projective variety, the situation is more complicated to handle.  
Let $X$ be a smooth projective variety defined over a number field $K$ of dimension $n>1$,  equipped with an arithmetic polarization  $M$ (as explained in Section \ref{basic arakelov}). We fix an embedding $\sigma :K\hookrightarrow \C$ and an analytic map $\varphi : \C\to X_{\sigma}(\C)$. For every $(r,H)\in\R^2$, denote by $C_\varphi(r,H)$ the cardinality of the set $S_\varphi(r,H):=\{ w\in\Delta_r\setminus E_r\;/\; \varphi(w)\in X(K)\;{\rm and}\; h_M(\varphi(w))\leq H\}$ where $E_r\subset\Delta_r$  is an exceptional set,which depends only on $\varphi$,  which is union of disks the sum of whose radii is less than ${{5}\over{T_{\varphi, M}(r)}}$.

The main theorem that we can prove in this case is the following:

\begin{theorem}\label{BPproj intro}(Cf. Theorem \ref{BPprojective}) In the hypothesis above, let $\epsilon>0$, then
\begin{equation}
C_\varphi(r,H)\ll T_{\varphi, M}((2+\epsilon)r)\cdot\exp(\epsilon(H+\log(T(r))\cdot T_{\varphi, M}(er)))
\end{equation} 
where the constants involved in $\ll$ depend only on $X$, $M$, $\epsilon$ and $\varphi$ but they are independent on $H$ and $r$.
\end{theorem}

Unfortunately, the most frustrating part of this statement is the presence of the exceptional set $E_r$ which, in principle may contain the preimage of every rational point.  We hope that, in a future, a more advanced study of the problem proposed in section \ref{basepoints} may give an improvement of Theorem \ref{BPprojective}. 

A statement similar to Theorem \ref{poly bound thm intro} can be obtained and proved following the same lines of the proof of it (and it will have polynomial bounds in terms of the characteristic function similar to Theorem \ref{BPproj intro}). We leave to the interested reader the statement and the proof of it. 

As we can see, Theorems \ref{BPaffine intro} and \ref{BPproj intro} are non relatively compact versions of Theorem \ref{BPil} and Theorem \ref{poly bound thm intro} is a generalization of Theorem \ref{lang} under the hypothesis of Theorem \ref{BPaffine intro}. We decided to keep the language as geometrical as possible because we think that these Theorems may have potential applications to diophantine questions. For instance, in the paper \cite{CW} the authors prove that, given a countable subset  $B$ of a rationally connected projective variety $X$, we can construct an entire curve in $X$ with order of growth zero and containing it. In principle we can apply theorems of this paper to the subset of rational points of a rational connected variety but in order to obtain interesting results, we should be able to prove that points of  height bounded by $H$  are contained in the image of a disk with radius an explicit function of $H$. We hope to be able to come back to these possible applications in a future paper. 

\

\section{Notations and basic facts from Arithmetic geometry}\label{basic arakelov}

\

This section serves as a reference for the arithmetic part of the article. Not all the notations and tools introduced in this section will be explicitly used in the article, but we present them here for completeness and to allow the reader to verify, each time, that the introduced constants depend only on the specified parameters.

Let $K$ be a number field and $O_K$ be its ring of integers. We will denote by $M_K^\infty$ the set of infinite places of $K$. We fix a place $\sigma_0\in M_K^\infty$. 

Let $X_K$ be a projective variety of dimension $N$ defined over $K$

If $\tau\in M_K^\infty$  and $F$ is an object over $X_K$ ($F$ may be a sheaf, a divisor, a cycle...), we will denote by $X_\tau$ the complex variety $X_K\otimes_\tau\C$ and by $F_\tau$ the restriction of $F$ to $X_\tau$. 

A model $\cX\to\Spec(O_K)$ of $X_K$ is a  flat projective $O_K$ scheme whose generic fiber is isomorphic to $X_K$. Suppose that $L_K$  and $\cX$ are  respectively a line bundle over $X_K$ and a model of it; We will say that a line bundle $\cL$ over $\cX$ is a model of $L_K$ if its restriction to the generic fiber is isomorphic to $L_K$. 

Let $\cX$ be a model of $X_K$. A hermitian line bundle $\overline\cL$ is a couple $(\cL, \langle\cdot;\cdot\rangle_\sigma)_{\sigma\in M_K^\infty}$, where $\cL$ is a line bundle over $\cX$ and, for every $\sigma\in M_K^\infty$, $\langle\cdot;\cdot\rangle_\sigma$ is a continuous  hermitian product on $L_\sigma$, with the condition that, if $\sigma=\overline\tau$, then $\langle\cdot;\cdot\rangle_\sigma=\overline{\langle\cdot;\cdot\rangle}_\tau$. An hermitian vector bundle is defined similarly. 

If $s$ is a local section of $L_\sigma$ defined in a neighborhood of a point $z\in X_\sigma(\C)$, then we will denote by $\Vert s\Vert_\sigma$ the real number $\langle s,s\rangle_\sigma^{1/2}$.

If $X_K$ is a projective variety, it is easy to see that for every line  bundle $L_K$  on $X_K$, we can find an embedding $\iota : X_K\hookrightarrow P_K$,  where $P_K$ is a smooth projective variety and $L=\iota^\ast(M)$ with $M$ line bundle on $P_K$. A metric on $L_K$ is said to be smooth if it is the restriction of a smooth metric on $M$.

Let $(\cL, \langle\cdot;\cdot\rangle_\sigma)_{\sigma\in M_K^\infty}$ be a hermitian line bundle on a model $\cX$  of $X_K$. If $s\in H^0(X_K, L^d_K)$ is a non zero section, we will denote by $\log\Vert s\Vert$ the real number $\sup_{\tau\in M_K^\infty}\{ \log\Vert s_\tau\Vert_\tau\}$. 

\smallskip

The $\sup$ norm and the $L^2$ norms are comparable, thus, in each situation, we will use the more suitable:

\begin{itemize}
\item[--] Let $L$ be a hermitian ample line bundle on a projective variety $Z$ equipped with a smooth metric $\omega$.  Over $H^0(Z, L^d)$ we can define two natural norms:
 \begin{equation}
 \Vert s\Vert_{\sup}:=\sup_{z\in Z}\{\Vert s\Vert (z)\}\;\; {\rm and}\;\;\Vert s\Vert_{L^2}:=\sqrt{\int_Z\Vert s\Vert^2 \omega^n}.
 \end{equation}
These norms are comparable: we can find constants $C_i$ such that

\begin{equation}\label{gromov}
C_1\Vert s\Vert_{L^2}\leq \Vert s\Vert_{\sup}\leq C_2^d\Vert s\Vert_{L^2}.
\end{equation}
\end{itemize}

This statement (due to Gromov) is proved for instance in \cite{SABK} Lemma 2 p. 166 when $Z$ is smooth. The general statement can be deduced by taking a resolution of singularities (remark that the proof of \cite{SABK} Lemma 2 p. 166  do not require that $L$ is ample). 

\smallskip

We recall the following standard facts of Arakelov theory:

\begin{itemize}
\item[--] If $L$ is a hermitian line bundle over $\Spec(O_K)$ and $s\in L$ is a non vanishing section, we define
\begin{equation}
\widehat{\deg}(L):=\log(Card(L/sO_K))-\sum_{\sigma_\in M^\infty_K}\log\Vert s\Vert_\sigma.
\end{equation}
\end{itemize}
If $E$ is a hermitian vector bundle of rank $r$ on $\Spec(O_K)$, we define $\widehat{\deg}(E):=\widehat{\deg}(\wedge^{r} E)$ and the slope of $E$ is $\widehat{\mu}(E)={{\widehat{\deg}(E)}\over{r}}$.  
\begin{itemize}
\item[--] Within all the sub bundles of $E$ there is one whose slope is maximal, we denote by $\widehat{\mu}_{\max}(E)$ its slope. It is easy to verify that $\widehat{\mu}_{\max}(E_1\oplus E_2)=\max\{ \widehat{\mu}_{\max}(E_1), \widehat{\mu}_{\max}(E_2)\}$.

\end{itemize}

We will need the following version of the Siegel Lemma: 

\begin{lemma}\label{siegel} (Siegel Lemma) Let $E_1$ and $E_2$ be hermitian vector bundles over $O_K$. Let $f:E_1\to E_2$ be a non injective linear map. Denote by $m=rk(E_1)$ and $n=rk(Ker(f))$. Suppose that there exists a positive real constant $C$ such that:

a) $E_1$ is generated by elements of $\sup$ norm less or equal than $C$.

b) For every infinite place $\sigma$ we have $\Vert f\Vert_\sigma\leq C$ 

Then there exists an non zero element $v\in Ker(f)$ such that
\begin{equation}
\sup_{\sigma\in M^\infty_K}\{ \log\Vert v\Vert_\sigma\}\leq{{m}\over{n}}\log(C^2)+\left( {{m}\over{n}}-1\right) \widehat{\mu}_{\max}(E_2)+3\log(n)+A
\end{equation}
where $A$ is a constant depending only on $K$. 
\end{lemma}

A proof of this version of Siegel Lemma can be found in \cite{gasbarri}.

Let $\cX\to\Spec(O_K)$ be a projective arithmetic variety and $\overline\cL$ be an hermitian line bundle over it.  We will say that $\overline\cL$ is {\it arithmetically very ample} if:

i) the involved line bundle $\cL$ is relatively ample over $\cX$;

ii) the $O_K$--algebra $\oplus_{d\geq 0}H^0(\cX,\cL^d)$ is generated by $H^0(\cX,\cL)$;

iii) the normed $O_K$--module $H^0(\cX,\cL)$ is generated by sections of $\sup$ norm less or equal then one.

Observe that, if $\overline\cL$ is arithmetically very ample, then for every $d\geq 0$ the normed $O_K$--module $H^0(\cX,\cL^d)$ is generated by sections of $\sup$ norm less or equal then one.

\begin{definition} An {\rm arithmetic polarization $(\cX,\overline\cL)$ of $X_K$} is the choice of the following data:

a) An ample line bundle $L_K$ over $X_K$;

b) A projective model $\cX\to\Spec(O_K)$ of $X_K$ over $O_K$.

c) A arithmetically very ample line bundle hermitian line bundle $\overline\cL$ over $\cX$ which is a model of $L_K^{\otimes r}$ for a suitable positive interger $r$.

d)  For every $\tau\in M_K^\infty$, we suppose that the metric on $L_\tau$ is smooth and positive.
\end{definition}

Remark that, if $L_K$ is an ample line bundle over $X_K$, we can find a positive integer $d$ and an arithmetic polarization $(\cX,\overline\cL_1)$ such that $\overline\cL_1$ is a model of $L_K^d$.

\smallskip

\begin{itemize}

\item[--] Let $L/K$ be a finite extension and $O_L$ the ring of integers of $L$. An $L$--point of $X_K$ is a $K$--morphism $P_L:\Spec(L)\to X_K$. The set of $L$ points of $X_K$ is noted  $X_K(L)$. If $(\cX,\cL)$ is an arithmetic polarization of $X_K$, by the valuative criterion of properness, every $L$--point $P_L:\Spec(L)\to X_K$ extends uniquely to an $O_K$--morphism $P_{O_L}:\Spec(O_L)\to \cX$. In this case, $P_{O_L}^\ast(\cL)$ is a hermitian line bundle on $\Spec(O_L)$. We define the height of $P_L$ with respect to $\cL$ to be the real number $h_\cL(P_L):={{\widehat{\deg}(P_L^\ast(\cL))}\over{[L:\Q]}}$.

\item[--]  {\it Liouville inequality}:  Let $p\in X_K(K)$ be a rational point. Let $p_0\in X_{\sigma_0}$ be its image. Then, for every positive integer $d$ and  global section $s\in H^0(\cX, \cL^d)$ such that $s(p)\neq 0$ we have
\begin{equation}\label{liouville1} 
\log\Vert s\Vert_{\sigma_0}(p)\geq-[K:\Q]\left( d\cdot h_\cL(p)+\log^+\Vert s\Vert\right).
\end{equation}
where $\log^+(a)=\sup\{0,\log(a)\}$. For a proof of this form of Liouville inequality cf. \cite{gasbarri2} Theorem 3.1.

\end{itemize}

\

\section{Influence of the base point on the growing conditions}\label{basepoints}

\

When one introduces the characteristic function $T_\varphi(r)$ of a holomorphic map $\varphi:\Delta_R\to X$ from a disk (possibly of infinite radius) to a projective variety, one needs to fix a base point, usually the center of the disk. The choice of this point, usually has no relevance on the study of the grown conditions of the function or of  Nevanlinna properties of it in general. 

Nevertheless, if one wants to use tools from Nevanlinna Theory to estimate norms of global sections of line bundles in a specific point, the choice of the point may be relevant and one may need to estimate how the different constants vary by changing it. 

Even the definition of the characteristic function depends on the choice of the point, in particular because, in order to define it, one needs to integrate on varying domains centered on the point itself. 

If the base point is changed and the new base point is very near to the old one with respect to the radius $r$, the modification of the characteristic function $T_\varphi(r)$ is essentially irrelevant. But if the distance of the two points is comparable with $r$, then the characteristic function may, a priori, change a lot for values of the radius comparable to $r$.

To explain the problem, let's fix some notation. 

For every positive real number $r$ denote by $\Delta_r$ the open disk {\it centered in 0} and of radius $r$ and by $\overline\Delta_r$ its closure. We fix a point $w_0\in\C$ and an hermitian line bundle $M$ on $\C$. 

We consider the Green function
\begin{equation}
\begin{array}{rcl}
g_r(w_o;z):&\overline\Delta_r\longrightarrow [0;+\infty]\\
z&\longrightarrow \log\left\vert {{r^2-z\overline w_0}\over{r(z-w_0)}}\right\vert.\\
\end{array}
\end{equation}
The function $g_r(w_0;z)$ is harmonic in $\Delta_r\setminus\{ w_0\}$, we have that  $g_r(w_o;z)^{-1}(0)=S_r:=\{\vert z\vert=r\}$,   $g_r(w_o;z)^{-1}(\infty)=\{w_0\}$ and it satisfies the differential equation
\begin{equation}
dd^c(g_r(w_o;z))=\mu_{w_0,r}-\delta_{w_0}
\end{equation}
where $\delta_{w_0}$ is the Dirac measure concentrated in $w_0$ and $\mu_{w_0,r}$ is the Poisson measure on $S_r$ given by $Re{{z+w_0}\over{z-w_0}}\cdot {{d\theta}\over{2\pi}}$ (with $z=re^{i\theta}$ and $Re(\cdot)$ is the real part).

Let $X$ be an analytic variety,  $M$ be an hermitian line bundle on it and let $c_1(M)$ be the first Chern form of $M$. We denote by $\Vert\cdot\Vert_M$ the hermitian metric on $M$.

 Let $\varphi:\C\to X$ be an holomorphic function. The characteristic function of $\varphi$ and $M$ is 
 \begin{equation}
 T_{\varphi,M}(r):=\int_0^r{{1}\over{t}}\cdot\int_{\Delta_t}\varphi^\ast(c_1(M)).
 \end{equation}
 
 Let $s$ be a global section of $\varphi^\ast(M)$. For every $z\in\C$ we denote by $v_z(s)$ the multiplicity of $s$ in $z$. In order to simplify the exposition, we suppose that $s(0)\neq 0$. The Nevanlinna First Main Theorem holds:
 \begin{equation}\label{FMT}
 \int_{S_r}\log\Vert s\Vert(re^{i\theta}){{d\theta}\over{2\pi}}+ T_{\varphi,M}(r)=\sum_{\vert z\vert<r}v_z(s)\cdot\log\left\vert{{r}\over{z}}\right\vert+\log\Vert s\Vert(0).
 \end{equation}
 
Observe that $\log\left\vert{{r}\over{z}}\right\vert=g_r(0,z)$. 
\begin{definition} The characteristic function of $\varphi$ and $M$ {\rm with respect to $w_0$} to be the function
\begin{equation}
 T_{\varphi,M,w_0}(r):=\int_0^1{{1}\over{t}}\cdot\int_{\exp(-g_r(w_0,z))<t}\varphi^\ast(c_1(M)).
  \end{equation}
  \end{definition}
 For  the characteristic function of $\varphi$ and $M$ with respect to $w_0$ a theorem similar to the Nevanlinna First Main Theorem holds:
 \begin{equation}\label{FMTw}
 \int_{S_r}\log\Vert s\Vert(re^{i\theta})\mu_{w_0,r}+ T_{\varphi,M,w_0}(r)=\sum_{\vert z\vert<r}v_z(s)\cdot g_r(w_0,z)+\log\Vert s\Vert(w_0).
 \end{equation}
 Observe that $T_{\varphi,M,0}(r)=T_{\varphi,M}(r)$. 
 
 Thus, since  $\int_{S_r}\log\Vert s\Vert(re^{i\theta}){{d\theta}\over{2\pi}}$ and $ \int_{S_r}\log\Vert s\Vert(re^{i\theta})\mu_{w_0,r}$ are easy to compare,  if $r$ is very big with respect to $\vert w_0\vert$, formulas \ref{FMT} and \ref{FMTw} are comparable.
 
 On the other side, when $\vert w_0\vert$ is comparable to $r$, it is not easy to  estimate $T_{\varphi,M,w_0}(r)$ in terms of $T_{\varphi,M}(r)$. Consequently it is not easy to deduce informations from  formula  \ref{FMTw}  if we  only have informations about $T_{\varphi,M}(r)$.
 
 In this section we will show that,  when $X$ be quasi--projective variety and $L$  is a very ample line bundle on it equipped with a positive metric, we can deduce some informations from formula \ref{FMTw} and the knowledge of $T_{\varphi,M}(r)$ even when $\vert w_0\vert$ is comparable to $r$. 
 
 Let $X$ be quasi--projective variety and $M$ a very ample line bundle on it equipped with a positive metric.  Let $\varphi:\C\to X$  be an holomorphic map. If $s\in H^0(X; M)$  and $r>0$, we will denote by $\Vert s\Vert _r$ the real number $\sup_{\vert z\vert\leq r}\{\Vert \varphi^\ast(s)\Vert(z)\}$.

While the standard First Main Theorem tells us that, if there are many zeros of a section which are very near to the origin of $\C$, then the  value of the norm of the section at the origin must be small with respect to the sup norm and the characteristic function, we will prove similar estimates  replacing the origin by an arbitrary point in $\Delta_r$. The interesting fact is that this estimate is in terms of the characteristic function {\it with respect to 0} and not with respect to the point itself. 

The first Theorem in this direction is for {\it affine} varieties:

\begin{theorem}\label{generalFMT} Under the hypotheses above and for every $\epsilon>0$, suppose, moreover, that $X$ is an affine variety. Then we can find a constant $A$  (depending only on $\varphi$ and $\epsilon$, in particular independent from $r$) such that, for every $r>0$ and $w_0$ such that $\vert w_0\vert<r$ we have that
\begin{equation}
\sum_{\vert z\vert<r}v_z(s)\cdot g_r(w_0,z)+\log\Vert s\Vert(w_0)\leq \log\Vert s\Vert_{r}+A(T_{\varphi, M}(1+\epsilon)r)+1).
\end{equation}
\end{theorem}

Observe that this Theorem is a slight generalization of the classical estimate: if $f$ is an entire function, then
\begin{equation}
\log(\Vert f\Vert_r)\leq 3T_f(2r).
\end{equation}
(cf. \cite{La}, Theorem 2.5 page 170).

\begin{proof}  We first remark that, since $\mu_{w_0,r}$ is a positive measure and $\int_{S_r}\mu_{w_0,r}=1$ , we have that
\begin{equation}
\int_{S_r}\log\Vert s\Vert(re^{i\theta})\mu_{w_0,r}\leq \log\Vert s\Vert_{r}.
\end{equation}
Consequently the Theorem is consequence of the inequality above, Formula \ref{FMTw} and  Proposition \ref{comparing T} below.
\end{proof}
\begin{proposition}\label{comparing T} Under the hypotheses of Theorem \ref{generalFMT}, for every $\epsilon>0$; we can find   a constant  $A$, such that, for every $w\in\Delta_r$ we have
\begin{equation}
T_{\varphi, M, w}(r)\leq A\cdot( T_{\varphi, M}((1+\epsilon)r)+1).
\end{equation}
\end{proposition}
\begin{proof} We may suppose that $X=\C^N\subset\P^N$ and $M$ is $\cO(1)$ equipped with the Fubini--Study metric. 

In this case $\varphi(z)=(f_1(z),\dots, f_N(z))$ where the $f_j(z)$ are entire functions and
\begin{equation}
T_{\varphi, M,w}(r)={{1}\over{2}}\cdot\int_{S_1}\log(1+\sum_{j=1}^N\vert f_j(re^{i\theta})\vert^2)Re\left({{re^{i\theta}+w}\over{re^{i\theta}-w}}\right){{d\theta}\over{2\pi}}-{{1}\over{2}}\cdot\log(1+\sum_{j=1}^N\vert f_j(w)\vert^2).
\end{equation}
Let $w_1\in\overline{\Delta}_r$ such that ${{1}\over{2}}\cdot\log(1+\sum_{j=1}^N\vert f_j(w_1)\vert^2)=\max_{w\in\overline{\Delta}_r}\{{{1}\over{2}}\cdot\log(1+\sum_{j=1}^N\vert f_j(w)\vert^2)\}$. Since the function 
${{1}\over{2}}\cdot\log(1+\sum_{j=1}^N\vert f_j(w)\vert^2)$ is subharmonic, for every $R>r$ we have that
\begin{equation}
{{1}\over{2}}\cdot\log(1+\sum_{j=1}^N\vert f_j(w_1)\vert^2)\leq{{1}\over{2}}\cdot\int_{S_1}\log(1+\sum_{j=1}^N\vert f_j(Re^{i\theta})\vert^2)Re\left({{Re^{i\theta}+w_1}\over{Re^{i\theta}-w_1}}\right){{d\theta}\over{2\pi}}
\end{equation}
and $Re\left({{Re^{i\theta}+w_1}\over{Re^{i\theta}-w_1}}\right)\leq {{R+r}\over{R-r}}$ we obtain
\begin{equation}
 \begin{array}{rcl}
 T_{\varphi, M,w}(r)&\leq&{{1}\over{2}}\cdot\int_{S_1}\log(1+\sum_{j=1}^N\vert f_j(re^{i\theta})\vert^2)Re\left({{re^{i\theta}+w}\over{re^{i\theta}-w}}\right){{d\theta}\over{2\pi}}\\
 &\leq &{{1}\over{2}}\cdot\log(1+\sum_{j=1}^N\vert f_j(w_1)\vert^2)\\
 &\leq & {{R+r}\over{R-r}}\int_{S_1}\log(1+\sum_{j=1}^N\vert f_j(Re^{i\theta}\vert^2)\cdot{{d\theta}\over{2\pi}})\\
 &=&{{R+r}\over{R-r}}T_{\varphi,M}(R)+{{R+r}\over{R-r}}\cdot{{1}\over{2}}\cdot\log(1+\sum_{j=1}^N\vert f_j(0)\vert^2).
 \end{array}
 \end{equation}
 Thus take  $R=(1+\epsilon)r$ and $A=\max\{ 2/\epsilon, 2/\epsilon\cdot\log(1+\sum_j\vert f_j(0)\vert^2)\}$ and we obtain
 \begin{equation}
 T_{\varphi,M,w}(r)\leq A(T_{\varphi, M}((1+\epsilon)r)+1)
 \end{equation}
 which concludes the proof.  \end{proof}
 
 Observe that a similar proof holds if we replace $\C^N$ with a variety $X$ where the pull back to $\C$ of the positive line bundle $M$ is generated by a global section $s$ such that $\Vert s\Vert\leq 1$. Above, this section is the hyperplane at infinity. 
 
 Adapting the proof of Theorem \ref{comparing T}, we can obtain a uniform estimate for general functions and   $w_0$ belonging to a fixed compact set: 
 \begin{proposition}\label{easyFMT} Let $r_0>0$ and $\epsilon>0$. Let $X$ be a projective variety and $M$ a line bundle over it.  Suppose the metric on $M$ is positive. Let $\varphi:\C\to X$ be an entire curve on $X$. Then we can find positive constants $A_i:=A_i(r_0,\epsilon)$ such that, for every $w_0\in\Delta_{r_0}$ and every $r>r_0$,  we have
\begin{equation}
\left\vert T_{\varphi,M,w_0}(r)-A_1\cdot T_{\varphi,M}((1+\epsilon)r)\right\vert\leq A_2.
\end{equation}
\end{proposition} 
\begin{proof} Since the metric on $M$ is positive, we may suppose that $X=\P^N$ and $M=\cO(1)$ with the Fubini--Study metric. In this case, we may find $N+1$ entire functions $f_0(z),\dots, f_N(z)$ , non simultaneously vanishing, such that $\varphi(z)=[f_0(z):\cdot :f_N(z)]$. Moreover an application of the First Main Theorem gives that
\begin{equation}
T_{\varphi, M,w}(r)={{1}\over{2}}\cdot\int_{S_1}\log(\sum_{j=0}^N\vert f_j(re^{i\theta})\vert^2)Re\left({{re^{i\theta}+w}\over{re^{i\theta}-w}}\right){{d\theta}\over{2\pi}}-{{1}\over{2}}\cdot\log(\sum_{j=0}^N\vert f_j(w)\vert^2).
\end{equation}
Since the function $\log(\sum_{j=0}^N\vert f_j(z)\vert^2)$ is subharmonic, an estimate similar to the one used in the proof of Theorem \ref{generalFMT} gives, for $R\geq r$,
\begin{equation}
T_{\varphi, M, w_0}(r)\leq {{R+r}\over{R-r}}\left(T_{\varphi,M}(R)+{{1}\over{2}}\cdot\log(\sum_{j=1}^N\vert f_j(0)\vert^2)\right)-{{1}\over{2}}\cdot\log(\sum_{j=0}^N\vert f_j(w_0)\vert^2).
\end{equation}

The function ${{1}\over{2}}\cdot\log(\sum_{j=0}^N\vert f_j(z)\vert^2)$ is a continuous function on the compact set $\overline\Delta_{r_0}$, thus it is lower bounded over it. As before, taking $R=(1+\epsilon)r$ and reversing the role of $0$ and $w_0$ we obtain the proof of the Proposition.\end{proof}

Suppose that  $C$ is an algebraic curve inside $X$ and $s\in H^0(C, M^d\vert_C)$ is non zero and  $d$ is a positive integer. Then, the number of zeros of $s$ (counted with multiplicity) is given by $d$ times the degree of $C$ with respect to $M$. If $\varphi:\C\to X$ is an entire curve, $s\in H^0(X, M^d)$ is non zero and $r>0$, the number of zeros (counted with multiplicity) of the restriction of $\varphi^\ast(s)$ to $\Delta_r$ can be estimate in terms of $T_{\varphi, M}(r)$ {\it and the section itself}. The constants involved in the estimation depend, in general on $s$ and may vary in a complicated way if we change it. 
As a consequence of Proposition \ref{easyFMT} above, we obtain an estimate on the number of zeros of of the restriction of $\varphi^\ast(s)$ to $\Delta_r$ which depends only on $T_{\varphi, M}(r)$ and $d$ {\it but it is independent on $s$}.

We will denote by $\deg(\varphi^\ast(s))_r$ the number $\sum_{\vert z\vert<r}v_z(\varphi^\ast(s))$, where $v_z(\varphi^\ast(s))$ is the multiplicity at $z$ of $\varphi^\ast(s)$.

\begin{proposition}\label{zeros} Suppose that we are in the hypothesis above and $\epsilon>0$, then there are  constants $C_i=C_i(\varphi, M, d,\epsilon)$, depending only on $\varphi$, $M$, $d$ and $\epsilon$, {\rm but independent on $s\in H^0(X,M^d)$}, such that
\begin{equation}
\deg(\varphi^\ast(s))_r\leq C_1\cdot T_{\varphi,M}((1+\epsilon)r)+C_2.
\end{equation}
\end{proposition}
\begin{proof} We may suppose that the image of $\varphi$ is Zariski dense. Fix a small $r_0>0$. On the vector space $H^0(X, M^d)$ we have two norms: $\Vert s\Vert_{\infty, X}:=\sup_{z\in X}\{\Vert s\Vert(z)\}$ and $\Vert s\Vert_{\infty, r_0}:=\sup_{z\in \Delta_{r_0}}\{\Vert s\Vert(\varphi(z))\}$.  Since $H^0(X, M^d)$ is finite dimensional, these two norms are comparable, thus there are constants $B_i$ such that, for every $s\in H^0(X, M^d)$ we have
\begin{equation}
B_1\cdot \Vert s\Vert_{\infty, X}\leq \Vert s\Vert_{\infty, r_0}\leq B_2 \cdot \Vert s\Vert_{\infty, X}.
\end{equation}

In the sequel of the proof we may suppose that $r_0<{{\epsilon}\over{2}}\cdot r$. Since the function $\vert {{(1+\epsilon)r(z-w)}\over{(1+\epsilon)^2r^2-z\overline w}}\vert$ is a distance on the disk $\Delta_{(1+\epsilon)r}$ (bounded above by 1), there is a positive constant $C=C(\epsilon)$, depending only on $\epsilon$ such that, for every $w\in\Delta_{r_0}$ and $z\in\Delta_r$ we have 
\begin{equation}\label{estimate g}
g_{(1+\epsilon)r}(z,w)\geq C.
\end{equation}
Let $s\in H^0(X, M^d)\setminus\{0\}$. In order to estimate $\deg(\varphi^\ast(s))_r$, we may suppose that $\Vert s\Vert_{\infty, X}=1$. Consequently, there is a point $w_0\in\Delta_{r_0}$ such that $\Vert \varphi^\ast(s)\Vert(w_0)\geq B_1$.  Now we can apply the First Main Theorem for $w_0$ \ref{FMTw}  and obtain, using estimate \ref{estimate g}  for $\epsilon_1<\epsilon$,
 
\begin{equation}
\log(B_1)+ C\cdot \deg(\varphi^\ast(s))_r\leq T_{\varphi, M,w_0}((1+\epsilon_1)r)
\end{equation}
Proposition \ref{easyFMT} allows to conclude the proof.\end{proof}

 The case of meromorphic functions is more involved. In this case we cannot compare the order of growth with the supremum of the functions on a disk. Thus we will rely on this, classical estimate by Cartan, one proof of it can be found in (\cite{Le} Theorem 3 page. 76).
 \begin{theorem}\label{cartan} (Cartan) Let $\mu$ be a Borel measure over $\C$ of finite mass: $\mu(\C)=M$. We consider the following function
 \begin{equation}
 V(z):=\int_\C\log\vert z-\zeta\vert\mu(\zeta).
 \end{equation}
 Let $H$ be a real number such that $0<H<1$. Then
 \begin{equation}
 V(z)>M\log(H)
 \end{equation}
 outside a set $E$ of disks the sum of whose radii is less than $5H$.
 \end{theorem}
 
 In the case of projective varieties we obtain the following Theorem:
 
\begin{theorem}\label{generalFMT1} Under the hypotheses above, suppose, moreover, that $X$ is {\rm a projective variety}. Then, for every $r>0$ and $w_0$ such that $\vert w_0\vert<r$  and  which is outside an exceptional set $E_r$ which is union of disks the sum of whose radii is less than ${{5}\over {T_{\varphi, M}(r)}}$, we have that
\begin{equation}
\sum_{\vert z\vert<r}v_z(s)\cdot g_r(w_0,z)+\log\Vert s\Vert(w_0)\leq \log\Vert s\Vert_{r}+(\log(T_{\varphi, M}(r))+\log(r)+\log(2))\cdot T_{\varphi, M}(e\cdot r).
\end{equation}
\end{theorem}

If one reads carefully the proof of Theorem \ref{cartan} in \cite{Le}, one sees that, if $r'>r$ then $E_{r'}\cap \Delta_r\subseteq E_r$. We think that, if $r$ is fixed, then, for $r'\gg r$, we have that $E_{r'}\cap\Delta_r=\emptyset$.

\begin{rem} It would be interesting to show that, given $r$, the value of the  smallest $r'$ such that $E_{r'}\cap\Delta_r=\emptyset$ depends only on $T_{\varphi,M}(r)$. \end{rem}

\begin{proof}  As in the affine case, the proof relies on  Formula \ref{FMTw} and  Proposition \ref{comparing T1} below. \end{proof}

\begin{theorem}\label{comparing T1} Under the hypotheses of Theorem \ref{generalFMT1} we can find an exceptional set $E_r\subset\Delta_r$ which is union of discs the sum of whose radii is less than ${{5}\over {T_{\varphi, M}(r)}}$, such that, if $w_0\in\Delta_r$ is outside it, then
\begin{equation}
T_{\varphi, M, w_0}(r)\leq (\log(T_{\varphi, M}(r))+\log(r)+\log(2))\cdot T_{\varphi, M}(e\cdot r).
\end{equation}
\end{theorem}

\begin{proof} We start by rewriting the characteristic function in a different way. We fix $w_0\in\Delta_r$ and denote by $G_r(z)$ the function $\left\vert{{r(z-w_0)}\over{r^2-z\overline w_0}}\right\vert$.

Let $B[r]:=\{ (z,y)\in \C\times [0;1] \; :\; \vert z\vert<y\}$, then, by Fubini Theorem and denoting by $p_1:B(r)\to\C$ et $p_2:B(R)\to [0;1]$ the natural projections, we have that 

\begin{equation}
 \begin{array}{rcl}
 T_{\varphi, M, w_0}(r)=&\int_0^1{{dt}\over{t}}\int_{G_r(z)<t}\varphi^\ast(c_1(M))\\
  =&\int_{B(r)}p_2^\ast({{dt}\over{t}})\wedge p_1^\ast(\varphi^\ast(c_1(M)))\\
  =&\int_{\Delta_r}\varphi^\ast(c_1(M))\int_{G_r(z)}^1{{dt}\over{t}}\\
  =&\int_{\Delta_r}g_r(w_0,z)\varphi^\ast(c_1(M)).\\
 \end{array}
 \end{equation}
 
 We will now use the following inequalities:
 \begin{itemize}
 \item[--] For every $w_0$ and $z$ in $\Delta_r$ we have $g_r(w_0, z)<g_{er}(w_0, z)$;
 \item[--] For every $z\in\Delta_r$ we have $\log{{er}\over{\vert z\vert}}>1$;
 \item[--] For every $w_0$ and $z$ in $\Delta_r$ we have $g_{er}(w_0 ,z)\leq \log\left({{(r(e+1)}\over{\vert z-w_0\vert}})\right)$.
 \end{itemize}
Thus we obtain
 \begin{equation}
 \begin{array}{rlc}
T_{\varphi, M, w_0}(r) =&\int_{\Delta_r}g_r(w_0,z)\varphi^\ast(c_1(M))\\
\leq&\int_{\Delta_r}g_{er}(w_0 , z)\log{{er}\over{\vert z\vert}}\varphi^\ast(c_1(M))\\
\leq& \int_{\Delta_r}\log(r(e+1))\log{{er}\over{\vert z\vert}}\varphi^\ast(c_1(M))-\int_{\Delta_r}\log(\vert z-w_0\vert)\log{{er}\over{\vert z\vert}}\varphi^\ast(c_1(M)).\\
\end{array}
\end{equation}
We have that 
\begin{equation}
 \int_{\Delta_r}\log(r(e+1))\log{{er}\over{\vert z\vert}}\varphi^\ast(c_1(M))\leq(2+\log(r))T_{\varphi, M}(er).
 \end{equation}
 Now we apply Theorem \ref{cartan} with $\mu(z)=\log^+{{er}\over{\vert z\vert}}\varphi^\ast(c_1(M))$ (where $\log+(x):=\sup(0,\log(x)$)) and $H={{1}\over{T_{\varphi, M}(r)}}$ consequently,
since $\int_{\Delta_r}\log{{er}\over{\vert z\vert}}\varphi^\ast(c_1(M))\leq T_{\varphi, M(er)}$,  we obtain that
\begin{equation}
\int_{\Delta_r}\log(\vert z-w_0)\vert\log{{er}\over{\vert z\vert}}\varphi^\ast(c_1(M))\geq -T_{\varphi, M}(er)\cdot\log(T_{\varphi,M}(r))
\end{equation}
for $w_0\in \Delta_r$ and outside an exceptional set $E_r$ which is union of disks the sum of whose radii is less than ${{5}\over {T_{\varphi, M}(r)}}$. The conclusion follows because we have that 
\begin{equation}
 \int_{\Delta_r}\log(r(e+1))\log{{er}\over{\vert z\vert}}\varphi^\ast(c_1(M))\leq(2+\log(r))T_{\varphi, M}(er).
 \end{equation}.
\end{proof}

\

\section{The general distribution of rational points in an expanding domain}

\

Let $X$ be either an affine or a projective variety defined over a number field $K$ with an ample line bundle $M$ equipped with a positive metric (in the affine case we suppose that $M$, and its metric, is the restriction of an hermitian line bundle on a smooth projective compactification of $X$). We will suppose that $M$ defines an arithmetic polarization of a a smooth projective compactification of $X$. We fix an embedding $\sigma :K\hookrightarrow\C$ and we denote by $X_\sigma$ the complex variety deduced from $X$ via the change of base $\sigma$. 

We fix an analytic map $\varphi:\C\to X_\sigma(\C)$ with Zariski dense image.

If $X$ is projective, for every $ r> 0$, let $E_r\subset\Delta_r$ be the exceptional set introduced in Theorem \ref{comparing T1}. We recall that $E_r$ is a union of disks whose sum of radii is less than ${{5}\over{T_{\varphi, M}(r)}}$. If $X$ is affine, we put $E_r=\emptyset$.

For every positive numbers $r$ and $H$, we define
\begin{equation}
S_\varphi(r,H):=\left\{ w\in \Delta_r\setminus E_r \;/\; \varphi(w)\in X(K) {\; \rm and\;} h_M(\varphi(w))\leq H\right\}.
\end{equation}
And we denote by $C_\varphi(r,H)$ its cardinality. 

In this section we will prove a generalization of the classical Bombieri - Pila Theorem \cite{BOMBIERIPILA} to expanding domains of the affine line. In the paper \cite{gas1} we proved a version of the Bombieri Pila Theorem for a relatively compact open set in a Zariski dense Riemann surface inside a projective variety. Although Bombieri Pila Theorem have been the starting point of a wide generalization on rational points of $O$--minimal sets, it, and its generalizations, cannot be applied, for instance, to  estimate the number of rational points on Riemann surfaces which are dense (for the Zariski topology) inside the variety. Here we will study the case of entire curves inside the varieties. We will find a theorem which is the analogue of Bombieri Pila's in this case.

As explained in the previous session, Schwarz Lemma for points which are near to the border of a disk has (at the moment) a different shape in the affine and in the projective case. For this reason, we will state two versions of the theorem, one in the affine and the other is the projective case. The proofs are very similar, the only difference will be that in one case we apply Proposition \ref{comparing T} and in the other we apply Theorem \ref{comparing T1}.

The affine version of the Theorem is:
\begin{theorem}\label{BPaffine} Suppose that $X$ is an affine  variety  of dimension $n>1$ equipped with an hermitian line bundle $M$ as above. Let $\varphi:\C\to X_\sigma(\C)$ be an analytic map with Zariski dense image. Let $\epsilon>0$. Then
\begin{equation}
C_\varphi(r,H)\ll T_{\varphi, M}((2+\epsilon)r)\cdot\exp(\epsilon(H+T_{\varphi, M}(1+\epsilon)r))
\end{equation} 
where the constants involved in $\ll$ depend only on $X$, $M$, $\epsilon$ and $\varphi$ but they are independent on $H$ and $r$.
\end{theorem}
The projective version of the Theorem is:
\begin{theorem}\label{BPprojective} Suppose that $X$ is projective variety of dimension $n>1$ equipped with an hermitian line bundle $M$ as above. Let $\varphi:\C\to X_\sigma(\C)$ be an analytic map with Zariski dense image. Let $\epsilon>0$. Then
\begin{equation}
C_\varphi(r,H)\ll T_{\varphi, M}((2+\epsilon)r)\cdot\exp(\epsilon(H+\log(T(r))\cdot T_{\varphi, M}(er)))
\end{equation} 
where the constants involved in $\ll$ depend only on $X$, $M$, $\epsilon$ and $\varphi$ but they are independent on $H$ and $r$.
\end{theorem}

\begin{rem} Observe that if,  in Theorems \ref{BPaffine} and \ref{BPprojective} above, we fix $r$, we find the version of the Bombieri--Pila Theorem given in \cite{gas1}.
\end{rem}

Before we start the proof, we need to introduce some other notation. We fix $r>0$ big enough and we denote by $r_1$ the number $(1+\epsilon)r$. We recall that the function
$d_{r_1}(z,w):=\vert{{r_(z-w)}\over{r_1^2-z\overline w}}\vert$ is a distance on $\Delta_{r_1}$. If $W\subset \Delta_r$ is a subset, we will denote by $diam_{r_1}(W)$ the number
$\sup\{ d_{r_1}(z,w)\; /\; z,w\in W\}$.

If $W\subset\Delta_r$, and $H>0$ we will denote by  $C_\varphi(r,H,W)$ the cardinality of $S_\varphi(r,H)\cap W$.

The key step of the proof of the Theorems above are the following Lemmas

\begin{lemma}\label{small w affine} Fix $d_0$ a sufficiently big integer. For every positive integer $d\geq d_0$, suppose, in the hypothesis of Theorem \ref{BPaffine}, that $W\subseteq\Delta_r$ is such that 
\begin{equation}
diam_{r_1}(W)\leq \exp\left(-{{T_{\varphi,M}((1+\epsilon)r)+H}\over{ d^{n-1}}}\right).
\end{equation}
Then, we can find a non zero section $s\in H^0(X,M^d)$ vanishing on every $w\in S_\varphi(r,H)\cap W$.\end{lemma}

And similarly, in the projective case, we have:

\begin{lemma}\label{small w proj} Fix $d_0$ a sufficiently big integer. For every positive integer $d\geq d_0$, suppose, in the hypothesis of Theorem \ref{BPprojective}, that $W\subseteq\Delta_r$ is such that 
\begin{equation}
diam_{r_1}(W)\leq \exp\left(- {{\log(T(r)\cdot T_{\varphi, M}(er)+H}\over{ d^{n-1}}}\right).
\end{equation}
Then, we can find a non zero section $s\in H^0(X,M^d)$ vanishing on every $w\in S_\varphi(r,H)\cap W$.\end{lemma}

\begin{proof}  The proof is similar to the proof of Lemma 4.5 of \cite{gas1}, but since there are some steps to change, we give here some details. 

Denote by $h^0(X,M^d)$ the rank of $H^0(X, M^d)$, it is well known that  can find a positive constant $B_1$ such that, for every $\epsilon_1>0$ and $d$ sufficiently big, we have 
\begin{equation}
B_1(1-\epsilon_1)d^n\leq h^0(X,M^d)\leq B_1(1+\epsilon_1) d^n.
\end{equation}
We also recall that, by definition of arithmetic polarization, we may suppose that $H^0(\cX,\cM^d)$ is generated by sections of norm less or equal than one. 

Fix an $\alpha>0$ and we suppose that $d$ is big enough to have that $\alpha h^0(X,L^d)>1$. 
Choose an integer $A$ such that $(1-2\alpha)h^0(X,L^d)\leq A\leq (1-\alpha)h^0(X,L^d)$ and a subset $Z_W(d)\subset S_\varphi(r,H)\cap W$ of cardinality $A$ (if the cardinality of $S_\varphi(r,H)\cap W$ is smaller then the Lemma easily follows from linear algebra). 

Denote by $E(H)$ the  $O_K$ module $\oplus_{z\in Z_W(T)}\cM^d|_{\varphi(z)}$. The rank of $E(H)$ is $A$ and $\mu_{\max}(E(H))\leq dH$.

We have a natural restriction map
\begin{equation}
\delta_H: H^0(\cX,\cM^d)\longrightarrow E(H).
\end{equation}

By Gromov theorem \ref{gromov}, if we put on $H^0(\cX, \cM^d)$ the $L_2$ hermitian structure and on $E(M)$ the direct sum hermitian structure, the norm of $\delta_H$ is bounded by $C_0d$ for a suitable constant $C_0$. 

Denote by $K(H)$ the kernel of $\delta_H$ and by $k(H)$ its rank. By construction we have that 
\begin{equation}
{{h^0(X,M^d)}\over{k(H)}}\leq{ {h^0(X,M^d)}\over {h^0(X,M^d)-(1-\epsilon)h^0(X,M^d)}}={{1}\over{\alpha}}.
\end{equation}
We may then apply Siegel Lemma \ref{siegel} and we obtain that there is a non vanishing section $s\in H^0(\cX, \cM^d)$ such that $\varphi^\ast(s)$ vanishes on every point of $E(H)$ and $\log\Vert s\Vert_{\sup}\leq C_3dH$ for a suitable constant $C_3$ independent on $H$ and $r$. 

Let $w\in S_\varphi(r,H)\cap W$ we will now prove that, there is $d_0>0$ such that, if $d\geq d_0$ then $\varphi^\ast(s)$ vanishes also on $w$.  

Suppose that $\varphi^\ast (s)$ do not vanish on $w$. By Liouville inequality \ref{***} we can find a  constant $C$ independent on $d$ such that
\begin{equation}
\log\Vert s\Vert(\varphi(w)\geq -CdH.
\end{equation}
Thus  the affine First Main Theorem \ref{generalFMT} implies that we can find positive constants $C_i$, independent on $r$ and $d$ such that
\begin{equation}
C_1d(H+T((1+\epsilon)r))\geq C_2\cdot d^n\cdot \log(-diam_{r_1}(W)).
\end{equation}
Since, by hypothesis, we suppose that 
\begin{equation}
diam_{r_1}(W)\leq \exp\left(-{{T_{\varphi,M}((1+\epsilon)r)+H}\over{ d^{n-1}}}\right),
\end{equation}
we obtain that
\begin{equation}
{{C_1}\over{d^{n-1}}}\geq C_2.
\end{equation}
And this is impossible as soon as $d$ is big enough, independently on $r$. \end{proof}
The proof of Lemma \ref{small w proj} is the same, one have to use to use the projective First Main Theorem \ref{generalFMT1} instead of Theorem \ref{generalFMT}.

 We need now to estimate how many subset  $W$ of $\Delta_r$ with $diam_{r_1}(W)\leq \exp\left(-{{T_{\varphi,M}((1+\epsilon)r)+H}\over{ d^{n-1}}}\right)$ are needed in order to cover $\Delta_r$.
 \begin{lemma}\label{covering} Let $\alpha>0$ be a real number strictly less than one. Let $N$ an integer strictly bigger than ${{5}\over{\alpha^2\cdot\epsilon}}$. Then we can find $N$ open subsets $W_1,\dots, W_N$, such that $diam_{r_1}(W_i)\leq \alpha$ and $\Delta_r\subset\cup_{i=1}^NW_i$.
\end{lemma}
\begin{proof} For every $a\in ]0;1[$  and $z\in \Delta_{r_1}$, denote by $B_{r_1}(z,a)$ the disk $\{ w\in\Delta_{r_1}\;/ d_{r_1}(z;w)\leq a\}$. 

For every $\alpha>0$, let  $N$ be  the number $\min\{ n\in\N\;/ \; \exists\; z_1,\dots, z_n \;{\rm in}\; \Delta_r \;{\rm s.t.}\; \Delta_r\subset \cup_{i=1}^nB_{r_1}(z_i,{{\alpha}\over{2}})\}$. 
If $M=\max\{ m\in \N;/ \exists\;  w_1,\dots,w_m \;{\rm in}\; \Delta_r \;{\rm s.t.}\; d_{r_1}(w_i,w_j)>{{\alpha}\over{2}}\}$, it is easy to see that $N\leq M$. Consequently, it suffices to estimate $M$ from above.

Before we estimate $M$ let's recall some basics of hyperbolic geometry: 
\begin{enumerate}
\item If $z_0\in \Delta_{r_1}$ and $s\in]0:1[$, we denote by $B_{r_1}(z_0,s)$ the ball $\{z\in\Delta_{r_1}\;:\; d_{r_1}(z_0,z)\leq s\}$. 
\item The hyperbolic pseudo distance $d_{r_1}(z;w)$ verify the strong triangle inequality: for every triple $z_1,z_2,z_3\in \Delta_{r_1}$ we have
\begin{equation}
d_{r_1}(z_1,z_2)\leq{{d_{r_1}(z_1,z_3)+d_{r_1}(z_3,z_2)}\over{1+d_{r_1}(z_1,z_3)d_{r_1}(z_3,z_2)}}
\end{equation}
\item The measure $\omega_{hyp}:={{i}\over{2}}\cdot{{dz\wedge d\overline{z}}\over{(r_1^2-\vert z\vert^2)^2}}$ is invariant under the M\"obius transformations of the disk $\Delta_{r_1}$.
Consequently the hyperbolic area of $B_{r_1}(z_0,s)$ is $A(s)={{\pi s^2}\over{1-s^2}}$. 
\end{enumerate}
Denote by $R$ the number 
\begin{equation}
R:={{{1}\over{1+\epsilon}}+{{\alpha}\over{2}}\over{1+{{\alpha}\over{2(1+\epsilon)}}}}
\end{equation}
Since $\Delta_r=B_{r_1}(0,{{1}\over{1+\epsilon}})$ and the strong triangle inequality (2) holds, we obtain that, for every $z_1\in\Delta_r$, we have that
\begin{equation}
B_{r_1}(z_1,{{\alpha}\over{2}})\subset B_{r_1}(0,R).
\end{equation}

Let $w_1, \dots,w_m$ be a subset of $\Delta_r$ such that $d_{r_1}(w_i,w_j)>{{\alpha}\over{2}}$. Then the balls $B_{r_1}(w_j,{{\alpha}\over{2}})$ are disjoint and all contained in  $B_{r_1}(0,R)$. Consequently $m\cdot A({{\alpha}\over{2}})\leq A(R)$, which gives
\begin{equation}
m\cdot{{\left({{\alpha}\over{2}}\right)^2}\over {1-{\left({{\alpha}\over{2}}\right)^2}}}\leq {{R^2}\over{1-R^2}}.
\end{equation}
From this we obtain
\begin{equation}
m\leq {{4-\alpha^2}\over{\alpha^2}}\cdot{{(2+(1+\epsilon)\alpha)^2}\over{(2(1+\epsilon)+\alpha)^2-(2+(1+\epsilon)\alpha)^2}}\leq{{5}\over{\alpha^2\cdot\epsilon}}
\end{equation}
\end{proof}

We can now give the proof of Theorem \ref{BPaffine} (and of Theorem \ref{BPprojective} which is similar).

 \begin{proof}{\it (Of Theorem \ref{BPaffine})} Choose $d$ such that $d^{n-1}\cdot\epsilon>1$. Let $N$ be an integer strictly bigger than  ${{5}\over{\epsilon}}\cdot \exp\left(2\epsilon\cdot(T_{\varphi,M}((1+\epsilon)r)+H))\right)$. By Lemma \ref{covering} we can cover $\Delta_r$ by $N$ disks $W_i$ such that  $diam_{r_1}(W_i)\leq \exp\left(-{{T_{\varphi,M}((1+\epsilon)r)+H}\over{ d^{n-1}}}\right)$. By Lemma \ref{small w affine}, for each disk $W_i$, we can find a non zero section $s_i\in H^0(X,M^d)$ vanishing on every $w\in S_\varphi(r,H)\cap W_i$. By Proposition \ref{zeros}, applied to each $s_i$ (remark that $d$ is fixed and depends only on $\epsilon$), the cardinality of $S_\varphi(r,H)\cap W_i$ is bounded by $C_1\cdot T_{\varphi, M}((1+\epsilon)r)$ for a suitable constant $C_1$ independent on $r$ and $H$. The conclusion of the proof follows. \end{proof}`
 
 \
 
 \section{Polynomial bounds for height of rational points}\label{polynomial bounds}
 
 \
 
 Examples by Surroca \cite{surroca} and by Boxall and Jones \cite{BJ} show that Theorem \ref{BPaffine} is optimal, in the sense that we cannot hope that the growth of the number of rational points of bounded height is less than exponential in $T((1+\epsilon)r)$ and $H$ . Nevertheless, in this session we will prove that there are infinitely many sets in $\R^2$ as big as we want, such that, if $T((1+\epsilon)r), H)$ is in one of these sets, then $C_\varphi(r,H)$ is less than a polynomial in $T((1+\epsilon)r)$ and $H$ of controlled degree. 
 
 In order to quantify these sets we will now give some definitions. We will state and prove the theorems for holomorphic maps in affine varieties. One can find similar statements for morphisms in projective varieties but these are less clean, even if, {\it mutatis mutandis}, the proof is the same. We leave to the interested reader the statement and the proofs of these.

 \begin{definition} Suppose that $X$ is an affine  variety  of dimension $m>1$ equipped with an hermitian line bundle $M$ as above. Let $\varphi:\C\to X_\sigma(\C)$ be an analytic map with Zariski dense image. Let $\epsilon>0$ and $\gamma>0$. We will denote by $L(\varphi, \epsilon, \gamma)\subset\R^2$ the following set
 \begin{equation}
 L(\varphi, \epsilon):=\left\{ (T_{\varphi, M}(r),H)\in\R^2\;\;/\;\; C_\varphi(r,H)\leq \epsilon\cdot(T_{\varphi,M}((1+\epsilon)r)+H)^\gamma\right\}.
 \end{equation}
 \end{definition}
 Observe that, since $T_{\varphi,M}(r)$ is a strictly increasing function, the value of it determines $r$.
 
 In this section, we provide $(\R_{\geq 0})^2$ with the norm $\Vert(x,y)\Vert=x+y$. It is not difficult to see that the results we present are independent on the chosen norm. 
 
 \begin{definition}\begin{itemize}\item Given a subset $U\subset(\R_{\geq 0})^2$ we will denote by $m_U$ the real number $m_U:=\min\{ \Vert (x;y)\Vert\;/\; (x,y)\in U\}$.
 \item Let  $(U_n)$ be a sequence of subsets of $(\R_{\geq 0})^2$, we will say that $(U_n)$ is disjointly unbounded if  $U_{n+1}\cap\bigcup_{i=1}^nU_i=\emptyset$ and $m_{U_n}\to\infty$.
 \item Let $A>1$ be a positive number and $(U_n)$ be a sequence of subsets of $(\R_{\geq 0})^2$, we will say that $(U_n)$ is $A$--bounded if $diam(U_n\cup U_{n-1})\leq A\cdot m_{U_{n-1}}$.
 \end{itemize} \end{definition}
 
 We should imagine a disjointly bounded and $A$--bounded sequence of subsets as a sequence of disjoints sets which become bigger and bigger but such that each of them is not so far from the the set before it. A typical example is the sequence of the open disks of radius $n$ and centered in the points $(n^2,0)$ which is a disjointly bounded $3$--bounded sequence for $n\geq 2$, and even better, for every $\epsilon>0$, there is $n_0$ such that, this sequence is disjointly bounded and $\epsilon$--bounded  if $n\geq n_0$.
 
 \begin{definition} A subset $B\subset(\R_{\geq 0})^2$ is said to be $A$--small if, for every disjointly unbounded and $A$--bounded sequence of subsets $(U_n)$ we can find $n_0$ such that $U_{n_0}\cap B=\emptyset$. A subset $W$ is said to be $A$--wide, if the complementary set $(\R_{\geq 0})^2\setminus W$ is $A$--small.\end{definition}
 Observe that, considering subsequences,  if $B$ is $A$--small and $(U_n)$ is a disjointly unbounded and $A$--bounded sequence, then $U_n\cap B=\emptyset$ for infinitely many $n$'s. Similarly, $W$ is $A$--wide, then, infinitely many $U_n$'s are contained in W. 
 
 In particular, an $A$-wide set $W$ contains infinitely many open disks of arbitrarily big radius.  
 
 We will now provide an explicit sufficient condition, for a subset $B\subset(\R_{\geq 0})^2$, to be $A$--small. 

\begin{definition} A sequence $((x_n,y_n))_{n\in\N}$ is said to be $A$-subgeometric if $\Vert(x_n, y_n)\Vert\to\infty$ and, for $n\geq n_0$, we have  $\Vert(x_n, y_n)\Vert\leq A\cdot\Vert(x_{n-1}, y_{n-1})\Vert$.
\end{definition}
\begin{lemma}\label{A small sets} A subset $B\subset(\R_{\geq 0})^2$  is $A$--small if it does not contain any  $(A+1)$-subgeometric sequence.
\end{lemma}
Another way to state this lemma is:  $B\subset(\R_{\geq 0})^2$ is $A$--small if every  $(A+1)$-subgeometric sequence has a subsequence disjoint from $B$.
\begin{proof} Suppose that $B$ is not $A$--small. Then there exists a disjointly unbounded and $A$--bounded sequence of subsets $(U_n)$ such that $U_n\cap B\neq\emptyset$ for every $n$. Let $(x_n,y_n)\in U_n\cap B$, then the sequence $(x_n,y_n)$ is contained in $B$ and it is $A$--subgeometric: 

-- $\Vert(x_n,y_n)\Vert\to\infty$ because $m_{U_n}\to\infty$;

-- $\Vert(x_n,y_n)\Vert= \Vert(x_n-x_{n-1}+x_{n-1};y_n-y_{n-1}+y_{n-1})\Vert\leq (A+1)\cdot\Vert(x_{n-1},y_{n-1})\Vert$ because $(U_n)$ is $A$--bounded.
\end{proof} 

The main Theorem of this session is:
\begin{theorem}\label{poly bound thm} Suppose that $X$ is an affine  variety  of dimension $m>1$ equipped with an hermitian line bundle $M$ as above. Let $\varphi:\C\to X_\sigma(\C)$ be an analytic map with Zariski dense image. Let $\epsilon>0$ and $\gamma>{{m}\over{m-1}}$. For every positive real number $A$, the set $L(\varphi, \epsilon,\gamma)$ is $A$--wide.
\end{theorem}
 \begin{proof} Let $B:=(\R_{\geq 0})^2\setminus L(\varphi, \epsilon)$. We have to prove that $B$ is $A$-small. 
 
 Suppose that $B$ is not $A$-small, then, by Lemma \ref{A small sets}, it contains a $(A+1)$-subgeometric sequence. More explicitly, if we suppose that $B$ is not $A$--small, we can find a sequence $(r_n, H_n)$ such that:
 
 a) $T_{\varphi,M}((1+\epsilon)r_n)+H_n\to\infty$;
 
 b) $T_{\varphi,M}((1+\epsilon)r_n)+H_n\leq (A+1)\cdot(T_{\varphi,M}((1+\epsilon)r_{n-1})+H_{n-1})$;
 
 c)  $C_\varphi(r_n,H_n)\geq \epsilon\cdot(T_{\varphi,M}((1+\epsilon)r_n)+H_n)^\gamma$.
 
 We suppose that $\epsilon\cdot(T_{\varphi,M}((1+\epsilon)r_1)+H_1)$ is big enough. Fix $a>0$ such that ${{m}\over{m-1}}\cdot(1+a)<\gamma$ and take an integer $d_1$ such that 
 \begin{equation}
 (\epsilon\cdot(T_{\varphi,M}((1+\epsilon)r_1)+H_1))^{{1+a}\over{m-1}}\leq d_1\leq (\epsilon\cdot(T_{\varphi,M}((1+\epsilon)r_1)+H_1))^{{1+a}\over{m-1}}.
 \end{equation}
 We also fix a subset $R_1\subseteq S_\varphi(r_1,H_1)$ of cardinality $A_1$ where $A_1$ is an integer verifying $(1-2\epsilon)h^0(X,M^{d_1})\leq A_1\leq (1-\epsilon)h^0(X,M^{d_1})$.  Since $M$ is ample, there is a constant $B$ such that $Bd_1^m\leq h^0(X,M^{d_1})\leq Bd_1^m+1$ as soon as $d_1$ is big enough.
 
 Using Siegel Lemma \ref{siegel}  in the same way as in the proof of Lemma \ref{small w affine}, we can construct a non vanishing  integral section $s\in H^0(X,M^{d_1})$ such that,$\varphi^\ast(s)$ vanishes on all the points of $R_1$ and such that $\log\Vert s\vert\leq C_1\cdot H_1\cdot d_1$, where $C_1$ is a positive constant, independent on $r$ and $H$.
 
 We now prove that $\varphi^\ast(s)$ vanishes in every point of $S_\varphi(r_1,H_1)$. 
 
 Let $w_0$ be a point of $S_\varphi(r_1,H_1)$ and suppose that $\varphi^\ast(s)(w_0)\neq 0$. By Liouville inequality \ref{***}, we have that
 \begin{equation}
 \log\Vert s(\varphi)(w_0)\Vert\geq -C_2\cdot d_1\cdot H_1.
 \end{equation}
 Where, $C_2$ is independent on $r$ and $H$. Consequently, since, there is a positive constant $C_3$ independent on $r$ such that, for every $z\in \Delta_r$ we have that $g_{(1+\epsilon)r}(w_0,z)\geq C_3$ and $\varphi^\ast(s)$ vanishes on $R_1$, Theorem \ref{generalFMT} implies that we can find constants $C_i$ independent on $r$, $H$ and $w_0$ such that 
 \begin{equation}
 C_1\cdot H_1\cdot d_1+d_1\cdot T_{\varphi, M}((1+\epsilon)r_1)\geq C_3\cdot A_1-C_2\cdot d_1\cdot H_1
 \end{equation}
 which, once we substitute the value of $d_1$ and $A_1$ gives
 \begin{equation}
 C_4(T_{\varphi, M}((1+\epsilon)r_1)+H_1)^{{{1+a}\over{m-1}}+1}\geq C_5\cdot (T_{\varphi, M}((1+\epsilon)r_1)+H_1)^{{{m(1+a)}\over{m-1}}}
 \end{equation}
 where, again, the $C_i$'s are independent on $r$ and $H$. Consequently, for $(T_{\varphi,M}((1+\epsilon)r_1)+H_1)$ large enough, we find a contradiction. Thus $\varphi^\ast(s)(w_0)=0$ for every $w_0\in S_\varphi (r_1,H_1)$. 
 
 By induction on $n$, we prove now prove that,  for every $n$, the section $\varphi^\ast(s)$ vanishes on the points of $S_\varphi (r_n,H_n)$.
 
 We may suppose that $\varphi^\ast(s)$ vanishes on all the points of $S_\varphi (r_{n-1},H_{n-1})$.
 
 Let $w_1\in S_\varphi (r_n,H_n)$ and suppose that $\varphi^\ast(s)(w_1)\neq 0$.
 
 Again, by Liouville inequality, we can find a constant $D_1>0$, independent on $n$ such that $\log\Vert s(\varphi)(w_1)\Vert\geq -D_1\cdot H_n$. And, again by Theorem \ref{generalFMT}, we can find constants $D_i$ independent on $n$ such that
 \begin{equation}
 D_2\cdot(T_{\varphi,M}((1+\epsilon)r_n)+H_n)\geq D_3\cdot C_\varphi(r_{n-1}, H_{n-1}). 
  \end{equation}
 But since we suppose that the sequence satisfies conditions (b)  and (c) above, we find constants $D_i$, independent on $n$, such that
 \begin{equation}
 D_4\cdot (T_{\varphi,M}((1+\epsilon)r_{n-1})+H_{n-1})\geq (T_{\varphi,M}((1+\epsilon)r_{n-1})+H_{n-1})^\gamma.
\end{equation}
 and this is impossible as soon as we suppose all the $(T_{\varphi,M}((1+\epsilon)r_n)+H_n)$'s are big enough. 
 
 Now, for every $n$, the section   $\varphi^\ast(s)$ should vanish on all the points of $S_\varphi (r_n,H_n)$. If the $T_{\varphi,M}(r_n)$ are bounded, this means that the $r_n$ are bounded and $\varphi^\ast(s)$ cannot have infinitely man zeros on $\Delta_R$ for a fixed $R$. If the $T_{\varphi,M}(r_n)$'s are not bounded, then the lower bound (c) of the sequence is in contradiction with Theorem \ref{zeros}. Thus such a section cannot exist.

Consequently the sequence $(r_n, H_n)$ verifying (a), (b) and (c) as above cannot exist and $B$ has to be $A$--small. The conclusion follows. \end{proof}

\end{document}